\documentclass[a4paper, reqno]{amsart}
\pdfoutput=1
\usepackage[english]{babel}
\usepackage{amssymb, upref}
\usepackage{enumerate}
\usepackage[colorlinks=true]{hyperref}
\hypersetup{urlcolor=blue, citecolor=red}

\title[Local uniform convergence and eventual positivity]{Local uniform convergence and eventual positivity of solutions to biharmonic heat equations}

\author[D. Daners]{Daniel Daners}
\address{Daniel Daners, School of Mathematics and Statistics, University of Sydney, NSW 2006, Australia} \email{daniel.daners@sydney.edu.au}

\author[J. Gl\"{u}ck]{Jochen Gl\"{u}ck}
\address{Jochen Gl\"{u}ck, University of Passau, Innstra{\ss}e 41, D-94032, Passau, Germany}
\email{Jochen.Glueck@uni-passau.de}

\author[J. Mui]{Jonathan Mui}
\address{Jonathan Mui, School of Mathematics and Statistics, University of Sydney, NSW 2006, Australia}
\email{jonathan.mui@sydney.edu.au}


\date{November 4, 2021}

\numberwithin{equation}{section}

\theoremstyle{plain}
\newtheorem{theorem}{Theorem}[section]
\newtheorem{proposition}[theorem]{Proposition}
\newtheorem{lemma}[theorem]{Lemma}
\newtheorem{corollary}[theorem]{Corollary}

\theoremstyle{definition}
\newtheorem{definition}[theorem]{Definition}

\newtheorem*{notation}{Notation}

\theoremstyle{remark}
\newtheorem{remark}[theorem]{Remark}
\newtheorem{assumption}[theorem]{Assumption}

\DeclareMathOperator{\sgn}{sign}
\DeclareMathOperator{\codim}{codim}

\newcommand{\mfra}{\mathfrak{a}}
\newcommand{\mfrb}{\mathfrak{b}}
\newcommand{\mfrq}{\mathfrak{q}}

\newcommand{\NN}{\mathbb{N}}

\newcommand{\RR}{\mathbb{R}}
\newcommand{\CC}{\mathbb{C}}

\begin{document}
\begin{abstract}
  We study the evolution equation associated with the biharmonic operator on infinite cylinders with bounded smooth cross-section subject to Dirichlet boundary conditions. The focus is on the asymptotic behaviour and positivity properties of the solutions for large times. In particular, we derive the local eventual positivity of solutions. We furthermore prove the local eventual positivity of solutions to the biharmonic heat equation and its generalisations on Euclidean space. The main tools in our analysis are the Fourier transform and spectral methods.
\end{abstract}

\maketitle

\section{Introduction}

In this paper, we study solutions to the biharmonic heat equation on Euclidean space $\RR^N$ and on infinite cylinders. We are interested in the asymptotic behaviour and positivity properties of solutions, a subtle matter which has only been brought to light relatively recently. In contrast to the well-known positivity-preserving property of the second-order heat equation, the biharmonic heat equation is not positivity preserving. However, not all is lost and some weaker positivity property persists. Gazzola and Grunau showed in~\cite{GG-lep} that solutions to the biharmonic heat equation on $\RR^N$ display \emph{local eventual positivity}. More precisely, they showed that for every continuous, non-trivial, compactly supported initial function $u_0$ satisfying $u_0(x) \geq 0$ and every compact subset $K$ of $\RR^N$, the corresponding solution $u(t,x)$ is positive on $K$ after finite time. This property was analysed in greater detail by the same authors together with Ferrero in~\cite{FGG} and generalised to fourth-order semilinear equations. Quite recently, L.~Ferreira and V.~Ferreira showed in~\cite{FF-lep} that local eventual positivity is a feature of solutions to \emph{polyharmonic} evolution equations. It therefore appears that local eventual positivity is a natural property to study in connection with higher-order evolution equations.

We begin by stating the two main results, first for the whole space $\RR^N$ (Theorem~\ref{thm:BHE-Rn}) and then for infinite cylinders (Theorem~\ref{thm:BHE-cyl}). This will be followed by a brief overview of eventual positivity.

\subsection{Main results}
Let $\alpha > 0$ and consider the evolution equation
\begin{equation}
  \label{eq:BHE-Rn}
  \begin{aligned}
    \frac{\partial u}{\partial t} + (-\Delta)^{\alpha} u & = 0
                                                         &          & \text{on } (0, \infty) \times \RR^N \\
    u(0, x)                                              & = u_0(x)
                                                         &          & \text{on } \RR^N
  \end{aligned}
\end{equation}
on $\RR^N$ with initial data satisfying
\begin{equation}
  \label{eq:BHE1-u0}
  u_0 \in L^1(\RR^N) \cap L^2(\RR^N).
\end{equation}
The case of $\alpha=1$ corresponds to the classical heat equation and $\alpha=2$ to the biharmonic equation. For $\alpha \in (0, 1)$, the operator $(-\Delta)^\alpha$ is the non-local \emph{fractional Laplacian}. In our main theorem, we derive the asymptotic behaviour of localised, rescaled solutions.

\begin{theorem}
  \label{thm:BHE-Rn}
  Suppose that $\alpha > 0$ and let $u$ be the solution to the evolution equation~\eqref{eq:BHE-Rn} with initial datum~\eqref{eq:BHE1-u0}. For every $t > 0$ define
  \begin{equation}
    \label{eq:PHE-Rn-ct}
    c_t := \frac{(2\pi)^N}{M_\alpha}t^{N/2\alpha} \qquad \text{where } M_\alpha:= \int_{\RR^N} e^{-|s|^{2\alpha}} \,ds.
  \end{equation}
  Then for any compact set $K \subseteq \RR^N$ we have
  \begin{equation}
    \label{eq:PHE-Rn-blowup-limit}
    \lim_{t\to\infty}c_t u(t,x)
    =\int_{\RR^N} u_0(y)\,dy
  \end{equation}
  uniformly with respect to $x\in K$.
\end{theorem}

Intuitively, the numbers $c_t$ are rescaling factors which counteract the decay of the solution. Observe that the spectral bound of $(-\Delta)^{\alpha}$ is zero. It is part of the continuous spectrum and the constant function is intuitively an eigenfunction, but it does not lie in the space $L^2(\RR^N)$. The theorem tells us that the limit~\eqref{eq:PHE-Rn-blowup-limit} is essentially a projection of the initial condition $u_0$ onto the corresponding `eigenspace'.

In the special case of the biharmonic equation ($\alpha=2$), our result is of the same flavour as~\cite[Theorem 1.1]{FGG}. In particular, we show that the asymptotic profile of this blown-up solution is locally that of a constant function. In Section~\ref{sec:BHE-Rn}, we prove Theorem~\ref{thm:BHE-Rn} and also note that the method allows to cover slightly more general evolution equations associated with linear differential operators with constant coefficients.

In Section~\ref{sec:cyl}, we consider a result similar to that in Theorem~\ref{thm:BHE-Rn} for the biharmonic equation on infinite cylinders of the form $\RR \times \Omega$, where $\Omega \subseteq \RR^N$ is a bounded smooth domain satisfying a spectral condition. We derive the asymptotic behaviour of solutions $u = u(t,x,y)$ to the problem
\begin{equation}
  \label{eq:BHE-cyl}
  \begin{aligned}
    \frac{\partial u}{\partial t} + (-\Delta)^2 u            & = 0   &  & \text{on } (0, \infty) \times \RR \times \Omega                \\
    u(0, \cdot, \cdot)                                       & = u_0 &  & \text{on } \RR \times \Omega                                   \\
    u(t,x,\cdot) =\frac{\partial u}{\partial \nu}(t,x,\cdot) & = 0   &  & \text{on } \partial\Omega \text{ for all } x \in \RR, t \geq 0
  \end{aligned}
\end{equation}
where $\nu(y)$ is the outer unit normal to $\partial\Omega$, and the initial datum satisfies
\begin{equation}
  \label{eq:BHE-cyl-u0}
  u_0 \in L^1(\RR, L^2(\Omega)) \cap L^2 (\RR, L^2(\Omega)).
\end{equation}
This means that $u_0(x, \cdot) \in L^2(\Omega)$ for every $x \in \RR$, and that
\begin{equation*}
  \int_\RR \| u_0(x,\cdot) \|_{L^2(\Omega)} \,dx < \infty \quad \text{and} \quad \int_\RR \| u_0(x,\cdot) \|^2_{L^2(\Omega)} \,dx < \infty.
\end{equation*}
The boundary conditions in \eqref{eq:BHE-cyl} are fourth-order homogeneous Dirichlet boundary conditions and, due to their physical interpretation as models of clamped plates, are often called \emph{clamped boundary conditions}, see for instance~\cite[Section 1.1.2]{GGS}. We adopt this terminology as well.

We will make use of properties of the eigenvalue problem
\begin{equation}
  \label{eq:BH-eig}
  \begin{aligned}
    \Delta^2 \phi & = \lambda \phi                           &  & \text{in } \Omega         \\
    \phi          & = \frac{\partial \phi}{\partial \nu} = 0 &  & \text{on } \partial\Omega
  \end{aligned}
\end{equation}
Since the above problem is self-adjoint and has compact resolvent, the spectrum consists of a sequence of eigenvalues of finite algebraic multiplicity bounded from below and going to infinity. We call the lowest eigenvalue the \emph{principal eigenvalue}, and the corresponding eigenfunction a \emph{principal eigenfunction} if the principal eigenvalue is geometrically simple. This is the case in second order problems. However, we note that the principal eigenvalue of the Dirichlet biharmonic operator need not be geometrically simple. Some explicit examples are given in~\cite[Section 3]{SS20}.

For our result on infinite cylinders requires the following spectral condition.
\begin{assumption}
  \label{simple}
  Let $\Omega \subseteq \RR^N$ be a bounded domain with $C^\infty$ boundary. We assume that the biharmonic eigenvalue problem \eqref{eq:BH-eig} has an algebraically simple principal eigenvalue.
\end{assumption}

Assumption~\ref{simple} is not too restrictive, as it has been shown by Ortega and Zuazua~\cite{OZ} that the spectrum is \emph{generically simple}. Briefly stated, for any smooth domain $\Omega$, there exists an arbitrarily small domain perturbation such that all eigenvalues of the operator on the perturbed domain are simple. Here is our main convergence result on cylinders.

\begin{theorem}
  \label{thm:BHE-cyl}
  Let $\Omega$ be a domain satisfying Assumption~\ref{simple}, and consider the solution $u$ to the biharmonic equation~\eqref{eq:BHE-cyl} with initial datum~\eqref{eq:BHE-cyl-u0}. Let $e_1(\cdot)$ be an eigenfunction, normalised in $L^2(\Omega)$, corresponding to the principal eigenvalue of \eqref{eq:BH-eig}. Then there exist numbers $c_t > 0$ such that
  \begin{equation}
    \label{eq:BHE-cyl-blowup-limit}
    \lim_{t\to\infty}c_t u(t,x,y)
    =\int_{\RR}\int_\Omega u_0(\xi,\eta) e_1(\eta) \,d\eta\,d\xi \; e_1(y)
    \qquad \text{as } t \to \infty
  \end{equation}
  uniformly with respect to $(x,y)\in I \times \Omega$, for any compact interval $I$ in $\RR$.
\end{theorem}

Theorem~\ref{thm:BHE-cyl} shows that the asymptotic profile of the rescaled solution is constant in the `infinite direction' (i.e.\ for $x \in \RR$) while it resembles the eigenfunction $e_1(\cdot)$ along the cross-sections $\Omega$ of the cylinder. The intuition here is similar as in the full space case previously discussed. The numbers $c_t$ counteract the decay of the solution. The function $v(x,y)=e_1(y)$, which is constant in the $x$-direction, takes the role of an `eigenfuction' associated with the spectral bound of $(-\Delta)^2$ on $\RR\times\Omega$. The limit~\eqref{eq:BHE-cyl-blowup-limit} is the projection of the initial condition into the direction of that `eigenfunction'.

\begin{remark}
  It is well-known that a maximum principle cannot be expected in general to hold for higher-order elliptic operators, so the eigenfunction $e_1$ is not necessarily of one sign. We discuss this further in Section~\ref{sec:4}. We note that the limit~\eqref{eq:BHE-cyl-blowup-limit} is independent of whether we choose to work with $e_1$ or $-e_1$.
\end{remark}

\begin{notation}
  Throughout this paper, if $E$ and $F$ are function spaces over the same domain $\Omega \subseteq \RR^N$, we write $E \cap F(\Omega)$ as an abbreviation for $E(\Omega) \cap F(\Omega)$.  For a measurable function $f\colon \Omega \subseteq \RR^N \to \RR$, we write $f \gneqq 0$ to mean that $f(x) \geq 0$ for a.e.\ $x \in \Omega$ and $f$ is not almost everywhere equal to 0.
\end{notation}

\subsection{Background on eventual positivity}
This paper was originally inspired by results from the theory of positive operator semigroups, which is by now a classic topic in operator theory---see for example the monograph~\cite{AGG}. Key features of this theory include the issues of asymptotic behaviour and stability, which are intimately linked to the spectral theory of linear operators. While these topics are interesting in abstract settings, they often have concrete manifestations in the study of partial differential equations, where the theory is applied to the semigroups generated by differential operators. We refer the reader to~\cite{BFR} for an accessible, modern survey of the theory of positive operator semigroups with an emphasis on concrete applications. Generally speaking, the study of positive linear operators is a natural extension of the classical Perron-Frobenius theory of positive matrices to infinite dimensional Banach spaces.

It seems that the phenomenon of \emph{eventual positivity} for operator semigroups in finite dimensions has been known for more than a decade, see for instance~\cite{NT} and the references therein. In infinite dimensions, motivated by~\cite{DD14}, a systematic theory of eventually positive semigroups and resolvents was developed recently by Kennedy and two of the present authors. The papers~\cite{DGK2, DGK1} contain the foundation of the theory, and there have since been various refinements and extensions~\cite{DG17, DG18, DG18-2, ArG20}.  In~\cite[Sections 7, 8]{AG20} the reader will find a snapshot of some applications of the theory of eventual positivity. Further applications to the analysis of partial differential equations can, for instance, be found in~\cite[Section 7]{DKP}, while the reader may consult~\cite[Section 6]{GM} and~\cite[Section 5]{BGM} for recent applications to the study of differential operators on graphs.

Very recently, a systematic operator-theoretic treatment of locally eventually positive semigroups was initiated by Arora in~\cite{Ar21}, with applications to the study of various differential equations on bounded domains. In~\cite{AGRT}, this theory was applied to study a fourth-order differential equation on $\RR^N$ but equipped with a probability measure. In the present paper, on the other hand, we study a particular equation on unbounded domains with infinite measure. A key difference is that in our case the spectral bound is not a simple eigenvalue, but part of the continuous spectrum. This necessitates a completely different set of tools in the analysis: we can no longer use properties of the spectral projection associated with the spectral value $0$, which have been a key feature in many results about eventually positive semigroups so far. Yet, while our techniques differ considerably from the tools applied in earlier papers on eventual positivity, we still note that the limiting objects~\eqref{eq:PHE-Rn-blowup-limit} and~\eqref{eq:BHE-cyl-blowup-limit} can be interpreted as `local' versions of spectral projections.

\section{Local eventual positivity of solutions}
\label{sec:LEP}

\subsection{Local eventual positivity on Euclidean space}
As a straightforward consequence of Theorem~\ref{eq:BHE-Rn}, we obtain a qualitative local eventual positivity result for solutions to equation~\eqref{eq:BHE-Rn}.

\begin{theorem}
  Let $\alpha > 0$ be fixed, and consider the solution $u = u(t,x)$ to the evolution equation
  \begin{equation}
    \label{eq:PHE-Rn}
    \begin{aligned}
      \frac{\partial u}{\partial t} + (-\Delta)^\alpha u & = 0
                                                         &          & \text{on } (0, \infty) \times \RR^N \\
      u(0, x)                                            & = u_0(x)
                                                         &          & \text{on } \RR^N
    \end{aligned}
  \end{equation}
  with initial datum $u_0 \in L^1 \cap L^2(\RR^N)$ such that
  \begin{equation}
    \label{eq:u0-pos-mass}
    \int_{\RR^N} u_0(x) \,dx>0.
  \end{equation}
  Then for every compact set $K \subseteq \RR^N$, there exists $T \geq 0$ depending on $u_0$ and $K$ such that
  \begin{equation*}
    u(t,x) > 0 \qquad \text{for all } (t,x)\in[T,\infty)\times K.
  \end{equation*}
\end{theorem}
\begin{proof}
  Let $K \subseteq \RR^N$ be an arbitrary compact set. If~\eqref{eq:u0-pos-mass} holds, then by Theorem~\ref{thm:BHE-Rn} there exist numbers $c_t > 0$ such that
  \begin{equation*}
    c_t u(t,x) \longrightarrow \int_{\RR^N} u_0(x) \,dx > 0
  \end{equation*}
  as $t \to \infty$ uniformly with respect to  $x\in K$. In particular there exists $T>0$ such that $u(t,x)>0$ for all $x\in K$ and $t>T$ as claimed.
\end{proof}
\begin{remark}
  The above theorem implies in particular that the solution of~\eqref{eq:PHE-Rn} is locally eventually positive if $u_0\gneqq 0$. We note that $\alpha=1$ corresponds to the classical heat equation, in which case the solution is positive for all $t\geq 0$ if $u_0\geq 0$. For local operators, we recall from the general theory in~\cite[Theorem~2.1]{ABR90} that the second-order operators are the only case where positivity is possible. For the fractional Laplacian, that is for $\alpha\in (0,1)$, the positivity is obtained directly from the fractional heat kernels, as shown in~\cite[Section 2]{VAZ}. In all other cases we can only expect (local) eventual positivity, unless we restrict to special classes of initial conditions as shown in \cite{GMO}.
\end{remark}

Our above result is a qualitative counterpart to the results in~\cite{GG-lep} and~\cite{FF-lep}, where the authors work explicitly with the polyharmonic heat kernels. We avoid the explicit kernels and employ Fourier analysis instead. We also point out that our results admit a larger class of initial data than was previously considered in the literature. In particular, we do not require continuity nor compactly supported functions. As a trade-off, our qualitative approach does not provide an estimate of the \emph{time to positivity}, that is, the quantity
\begin{equation*}
  T = T(K) := \inf_{\tau > 0} \{ u(t,x) > 0 \text{ for all } x \in K \text{ and } t \geq \tau \}.
\end{equation*}

\subsection{Local eventual positivity on infinite cylinders}
\label{sec:4}
We now investigate local eventual positivity for solutions of the biharmonic heat equation on the infinite cylinder $\RR \times \Omega$, where $\Omega$ is a bounded domain in $\RR^N$ satisfying Assumption~\ref{simple}. Theorem~\ref{thm:BHE-cyl} shows that the asymptotic behaviour of the solution to the parabolic problem on $\RR \times \Omega$ is determined by the sign of a normalised principal eigenfunction of the biharmonic operator with clamped (i.e.\ homogeneous Dirichlet) boundary conditions on the cylinder cross-section. Thus we need to discuss briefly the corresponding elliptic problem. There is now an extensive body of research on the positivity properties of solutions to the biharmonic Dirichlet problem
\begin{equation}
  \label{eq:BH-bvp2}
  \begin{aligned}
    \Delta^2 u & = f                                   &  & \text{in } \Omega         \\
    u          & = \frac{\partial u}{\partial \nu} = 0 &  & \text{on } \partial\Omega
  \end{aligned}
\end{equation}
and more general polyharmonic boundary value problems. We mention in particular the monograph~\cite{GGS} and the many references therein. The question of positivity of the first eigenfunction of the problem~\eqref{eq:BH-bvp2} with respect to the domain is a highly delicate one. From the spectral theory of positive irreducible operators---see for instance~\cite[Theorem 43.8]{Zaa} for a result suitable to our situation---one obtains a strictly positive principal eigenfunction whenever the Green's function of~\eqref{eq:BH-bvp2} is strictly positive. In the case where $\Omega$ is a ball, strict positivity of the Green's function can be easily deduced using the explicit formula of Boggio~\cite[Section 4.1]{GGS}, which of course implies that the solution to~\eqref{eq:BH-bvp2} satisfies $u \gneqq 0$ whenever $f \gneqq 0$. It is known that this property is preserved for ``sufficiently small'' perturbations of the ball in dimenions $N \geq 2$. These results and various generalisations are collected in~\cite[Chapter~6]{GGS}---note in particular Theorem 6.3 for domains in $\RR^2$ and Theorem 6.29 for domains in $\RR^N$ with $N \geq 3$. Alternatively, the reader may consult~\cite[Theorem 2]{GR10} for the case $N \geq 3$.

We also mention that it is possible to obtain a strictly positive first eigenfunction on domains where the biharmonic Dirichlet problem~\eqref{eq:BH-bvp2} is not positivity preserving (i.e.\ $f \gneqq 0$ does not imply $u \geq 0$). This was shown by Grunau and Sweers in~\cite{GS99} (see Theorem 2 in particular). In light of the present discussion, we make the following general conclusions regarding local eventual (non-)positivity of the parabolic problem on infinite cylinders.

\begin{theorem}
  \label{thm:LEP-cyl}
  Let $u_0 \in L^1 \cap L^2(\RR, L^2(\Omega))$, where $\Omega$ is a bounded domain in $\RR^N$ satisfying Assumption~\ref{simple}. Let $e_1$ be a normalised principal eigenfunction of the boundary value problem~\eqref{eq:BH-eig}. Assume that $u_0$ satisfies
  \begin{equation}
    \label{eq:u0-pos-mass-cyl}
    \int_\RR \int_\Omega u_0(x,y) e_1(y) \,dx \,dy > 0
  \end{equation}
  and let $u = u(t, x, y)$ be the solution to the biharmonic heat equation~\eqref{eq:BHE-cyl} on $\RR \times \Omega$ with initial datum $u_0$.

  Let $K_0$ be the zero set of the eigenfunction $e_1$, i.e.
  \begin{equation*}
    K_0 = \{ y \in \Omega\colon e_1(y) = 0 \}.
  \end{equation*}
  For every compact subset $K \subseteq \Omega \setminus K_0$ and compact interval $I \subset \RR$, there exists $T \geq 0$ depending only on $I \times K$ and the initial datum $u_0$ such that
  \begin{equation}
    \label{eq:LEP-cyl}
    \sgn(u(t,x,y)) = \sgn e_1(y)
  \end{equation}
  for all $t \geq T$ and $(x,y) \in I \times K$.
\end{theorem}

\begin{proof}
  By Theorem~\ref{thm:BHE-cyl}, there exist numbers $c_t > 0$ such that
  \begin{displaymath}
    \lim_{t \to \infty} c_t u(t,x,y) = \int_\RR \int_\Omega u_0(\xi, \eta) e_1(\eta) \,d\xi \,d\eta \; e_1(y)
  \end{displaymath}
  uniformly with respect to $(x,y) \in I \times \Omega$. At every $y \in K$, $e_1(y)$ has a well-defined sign, and by the continuity there exists $\delta>0$ such that $|e_1(y)|>\delta$ for all $y\in K$. Now it follows from the assumption~\eqref{eq:u0-pos-mass-cyl} that there exists $T>0$ such that~\eqref{eq:LEP-cyl} holds for all $t\geq T$ and $(x,y)\in I \times K$.
\end{proof}

\begin{remark}
  If $\Omega$ is a domain on which the principal eigenfunction of problem~\eqref{eq:BH-eig} may be chosen strictly positive, then $K_0 = \emptyset$, and we have local eventual positivity. Stated more precisely, for every compact subset $K \subseteq \Omega$ and compact interval $I \subset \RR$ there exists a time $T \geq 0$ depending on $u_0$ and $I \times K$ such that
  \begin{equation*}
    u(t,x,y) > 0
  \end{equation*}
  for all $t \geq T$ and for all $(x,y) \in I \times K$.
\end{remark}

\section{The biharmonic heat equation on Euclidean space}
\label{sec:BHE-Rn}

\subsection{The initial value problem and its Fourier transform}

We use the following convention for the Fourier transform $\widehat{u}_0$ of $u_0$:
\begin{equation*}
  \widehat{u}_0(\omega)
  = \frac{1}{(2\pi)^{N/2}} \int_{\RR^N} u_0(x) e^{-i \omega \cdot x} \,dx.
\end{equation*}
Defined this way, the Fourier transform is an isometric isomorphism of $L^2(\RR^N)$ (Plancherel's theorem). Note that the assumption $u_0 \in L^1 \cap L^2(\RR^N)$ combined with the Riemann-Lebesgue Lemma yields
\begin{equation}
  \label{eq:PHE1-u0-hat}
  \widehat{u}_0 \in L^2\cap C_0(\RR^N).
\end{equation}
We begin by taking the spatial Fourier transform of the initial value problem~\eqref{eq:PHE-Rn}, thus obtaining
\begin{equation*}
  \begin{aligned}
    \frac{\partial}{\partial t}\widehat{u}(t,\omega)
     & = -|\omega|^{2\alpha} \widehat{u}(t,\omega), \\
    \widehat{u}(0, \omega)
     & = \widehat{u}_0(\omega).
  \end{aligned}
\end{equation*}
Solving this differential equation explicitly we see that
\begin{equation*}
  \widehat{u}(t, \omega) = e^{-t|\omega|^{2\alpha}} \widehat{u}_0(\omega).
\end{equation*}
It is clear from~\eqref{eq:PHE1-u0-hat} and the rapid decay of the kernel $e^{-t|\omega|^{2\alpha}}$ that $\widehat{u}(t, \cdot) \in L^1 \cap L^2(\RR^N)$ for each $t > 0$. Thus the solution to~\eqref{eq:PHE-Rn} can be obtained by the inverse Fourier transform
\begin{equation}
  \label{eq:PHE-Rn-sol}
  u(t, x) = (2\pi)^{-N/2}\int_{\RR^N} e^{-t|\omega|^{2\alpha}} \widehat{u}_0(\omega) e^{i \omega \cdot x} \,d\omega.
\end{equation}

\subsection{Asymptotic behaviour of solutions}

As shown already in Section~\ref{sec:LEP}, the local eventual positivity of solutions to~\eqref{eq:BHE-Rn} for $\alpha>0$ is a straightforward consequence of Theorem~\ref{thm:BHE-Rn}. The techniques we use in the proof of this theorem can be extended to more general differential operators, as we will show in  Section~\ref{subsec:further}. Moreover, the analysis on $\RR^N$ can be considered as setting the stage for the more complicated analysis on cylinders in Section~\ref{sec:cyl}.

The key ingredient for the proof of Theorem~\ref{thm:BHE-Rn} is the following observation. For $\omega\in\RR$ and $\alpha>0$ let
\begin{equation*}
  \varphi_1(\omega):=\frac{e^{-|\omega|^{2\alpha}}}{M_\alpha}
  \qquad\text{with}\quad
  M_\alpha:=\int_{\RR^N}e^{-|\omega|^{2\alpha}}\,d\omega.
\end{equation*}
and define
\begin{equation}
  \label{eq:approx-id}
  \varphi_t(\omega):=t^{N/2\alpha}\varphi_1(t^{1/2\alpha}\omega)
\end{equation}
for all $t>0$. By the change of variables $s = t^{1/2\alpha}\omega$, we see that
\begin{equation*}
  \int_{\RR^N} \varphi_t(\omega) \,d\omega
  = \int_{\RR^N} \varphi_1(s) \,ds
  = 1
\end{equation*}
for all $t>0$. Hence the family $(\varphi_t)_{t>0}$ is an \emph{approximate identity} as $t \to \infty$ in the following sense.

\begin{definition}
  \label{def:approx-id}
  An \emph{approximate identity} as $t \to \infty$ is a family of measurable functions $(\rho_t)_{t > 0}$ from $\RR^N$ to $\RR$ such that
  \begin{enumerate}[(i)]
  \item $\rho_t(\omega) \geq 0$ for almost every $\omega \in \RR^N$;
  \item $\int_{\RR^N} \rho_t(\omega)\,d\omega = 1$ for all $t > 0$; and
  \item $\int_{|\omega| \geq \delta} \rho_t(\omega)\,d\omega \to 0$ as $t \to \infty$ for each $\delta > 0$.
  \end{enumerate}
\end{definition}
One has the following standard result on convolution with approximate identities. We include a proof to accommodate an additional parameter.
\begin{lemma}
  \label{lem:approx-id}
  Let $(\rho_t)_{t > 0}$ be an approximate identity as $t \to \infty$ and let $U$ be a non-empty set. Suppose that $f\colon\RR^N\times U\to\RR$ is bounded and that $f(\cdot\,,x)$ is measurable for all $x\in U$. If $f(\cdot\,,x)$ is continuous at $\omega_0\in\RR^N$ uniformly with respect to $x\in U$, then
  \begin{equation*}
    \lim_{\omega\to \omega_0}(\rho_t * f(\cdot,x))(\omega_0)
    =f(\omega_0,x)
  \end{equation*}
  uniformly with respect to $x\in U$.
\end{lemma}
\begin{proof}
  Since $\rho_t\geq 0$ and $\|\rho_t\|_1=1$ we see that for every $\delta>0$
  \begin{equation*}
    \begin{split}
      |(\rho_t * f(\cdot,x))(\omega_0)&-f(\omega_0,x)|
      =\Bigl|\int_{\RR^N}\rho_t(\omega_0-\omega)\bigl(f(\omega,x)-f(\omega_0,x)\bigr)\,d\omega\\
      &\leq\int_{\RR^N}\rho_t(\omega_0-\omega)\bigl|f(\omega,x)-f(\omega_0,x)\bigr|\,d\omega\\
      &\leq\int_{|\omega-\omega_0|\geq\delta}\rho_t(\omega_0-\omega)\bigl|f(\omega,x)-f(\omega_0,x)\bigr|\,d\omega\\
      &\qquad+\int_{|\omega-\omega_0|<\delta}\rho_t(\omega_0-\omega)\bigl|f(\omega,x)-f(\omega_0,x)\bigr|\,d\omega\\
      &\leq 2\|f\|_\infty\int_{|\omega|\geq\delta}\rho_t(\omega)\,d\omega
      +\sup_{|\omega-\omega_0|<\delta}\bigl|f(\omega,x)-f(\omega_0,x)\bigr|.
    \end{split}
  \end{equation*}
  Let $\varepsilon>0$ be arbitrary. By the uniform continuity of $f(\omega,x)$ at $\omega_0$ with respect to $x\in U$ there exists $\delta>0$ such that
  \begin{equation*}
    \sup_{|\omega-\omega_0|<\delta}\bigl|f(\omega,x)-f(\omega_0,x)\bigr|<\frac{\varepsilon}{2}
  \end{equation*}
  for all $x\in U$. With this choice of $\delta>0$, using the definition of an approximate identity, there exists $t_0>0$ such that
  \begin{equation*}
    0\leq  2\|f\|_\infty\int_{|\omega|\geq\delta}\rho_t(\omega)\,d\omega<\frac{\varepsilon}{2}
  \end{equation*}
  for all $t>t_0$.  Putting everything together we see that
  \begin{equation*}
    |(\rho_t * f(\cdot,x))(\omega_0)-f(\omega_0,x)|
    <\frac{\varepsilon}{2}+\frac{\varepsilon}{2}
    =\varepsilon
  \end{equation*}
  for all $t>t_0$. As $\varepsilon>0$ was arbitrary this proves the lemma.
\end{proof}

We now present the proof of Theorem~\ref{thm:BHE-Rn}.

\begin{proof}[Proof of Theorem~\ref{thm:BHE-Rn}]
  Let $c_t$ be defined as in~\eqref{eq:PHE-Rn-ct} and let $(\varphi_t)_{t>0}$ be the approximate identity given by~\eqref{eq:approx-id}. Moreover, let $\check\varphi_t(x):=\varphi_t(-x)$ be the reflection of $\varphi_t$. (In this case $\check\varphi_t=\varphi_t$, but for the benefit of possible generalisations we keep the notation). For $x,\omega\in\RR^N$, we define
  \begin{equation*}
    f(\omega,x):=(2\pi)^{N/2}\widehat{u}_0(\omega)e^{i\omega\cdot x}.
  \end{equation*}
  Recall that the solution $u(t,x)$ is given by~\eqref{eq:PHE-Rn-sol}. Multiplying it by $c_t$ we see that
  \begin{equation}
    \label{eq:ctu-convolution}
    c_tu(t,x)
    =(2\pi)^{N/2}\int_{\RR^N}\varphi_t(\omega)\hat u_0(\omega)e^{ix\cdot\omega}\,d\omega
    =\bigl(\check\varphi_t*f(\cdot\,,x)\bigr)(0)
  \end{equation}
  for all $x\in\RR^N$. As $u_0\in L^1\cap L^2(\RR^N)$ we have $\widehat{u}_0\in L^2\cap C_0(\RR^N)$ and thus $f\in C(\RR^N\times\RR^N)$. If $K\subseteq\RR^N$ is compact and $\delta_0>0$, then $f\colon[-\delta_0,\delta_0]\times K\to\CC$ is uniformly continuous and hence $\omega\mapsto f(\omega,x)$ is continuous at $\omega=0$ uniformly with respect to $x\in K$. As $(\check\varphi_t)_{t>0}$ is an approximate identity, it follows from~\eqref{eq:ctu-convolution} and Lemma~\ref{lem:approx-id} that
  \begin{equation*}
    \lim_{t\to\infty}c_tu(t,x)
    =f(0,x)
    =(2\pi)^{N/2}\widehat{u}_0(0)
    =\int_{\RR^N}u_0(y)\,dy
  \end{equation*}
  uniformly with respect to $x\in K$, completing the proof of Theorem~\ref{thm:BHE-Rn}.
\end{proof}

\subsection{Possible generalisations}
\label{subsec:further}
Let us make some remarks about the validity of Theorem~\ref{thm:BHE-Rn} for general linear differential operators with constant coefficients. Using the standard multi-index notation
\begin{equation*}
  D^\alpha = \frac{\partial^{\alpha_1}}{\partial x_1} \frac{\partial^{\alpha_2}}{\partial x_2} \cdots \frac{\partial^{\alpha_N}}{\partial x_N},
  \qquad \alpha = (\alpha_1, \ldots, \alpha_N) \in \NN^N,
\end{equation*}
such an operator takes the form
\begin{equation*}
  P(D) := \sum_{0 \leq |\alpha| \leq d} c_\alpha (iD)^\alpha \qquad (c_\alpha \in \RR)
\end{equation*}
where $d$ is order of the operator and $|\alpha| = \sum_{i=1}^N \alpha_i$. Upon taking the Fourier transform, we obtain
\begin{equation*}
  \widehat{(iD^\alpha) u} = \omega^\alpha \widehat{u}
\end{equation*}
where $\omega^\alpha$ is a standard abbreviation for $\omega_1^{\alpha_1} \omega_2^{\alpha_2} \cdots \omega_N^{\alpha_N}$. Hence
\begin{equation*}
  \widehat{ P(iD)u } = P(\omega) \widehat{u}
\end{equation*}
and $P(\omega)$ is called the \emph{symbol} of the operator. The associated evolution equation is given by
\begin{equation}
  \label{eq:evol-gen}
  \begin{aligned}
    \frac{\partial u}{\partial t} + P(iD)u & = 0      &  & \text{on } (0, \infty) \times \RR^N \\
    u(0, x)                                & = u_0(x) &  & \text{on } \RR^N.
  \end{aligned}
\end{equation}
\begin{enumerate}[(i)]
\item In the simplest case, suppose that $P(\cdot)$ is a homogeneous polynomial of even order, say $d = 2m$ for some $m \geq 1$, with the following structure:
  \begin{equation*}
    P(\omega) = (-1)^m \sum_{|\alpha| = 2m} c_\alpha \omega^\alpha \qquad c_\alpha \geq 0 \text{ and not all }0.
  \end{equation*}
  After taking Fourier transforms, the evolution equation becomes
  \begin{equation*}
    \frac{d \widehat{u}}{dt} = -P(\omega) \widehat{u}.
  \end{equation*}
  With initial datum $u_0$, we obtain
  \begin{equation*}
    \widehat{u}(t, \omega) = e^{-t P(\omega)} \widehat{u}_0(\omega) \qquad t \geq 0, \omega \in \RR^N,
  \end{equation*}
  and thus, at least formally, the solution to the evolution equation is given by the Fourier inverse of the above. By virtue of the homogeneity of the polynomial $P$, we can set $s = t^{1/2m} \omega$ to obtain
  \begin{equation*}
    \int_{\RR^N} e^{-tP(\omega)} \,d\omega = t^{-N/2m} \int_{\RR^N} e^{-P(s)} \,ds.
  \end{equation*}
  Provided that $\int_{\RR^N} e^{-P(s)} \,ds < \infty$, the blow-up factors $c_t$ can therefore be defined by
  \begin{equation*}
    c_t := t^{N/2m} \left(\int_{\RR^N} e^{-P(s)} \,ds \right)^{-1}.
  \end{equation*}
  If it can be shown that $\varphi_t(\omega) = c_t e^{-t P(\omega)}$ defines an approximate identity in the sense of Definition~\ref{def:approx-id}, then the techniques used in the proof of Theorem~\ref{thm:BHE-Rn} can be extended to this situation.

\item If we include lower-order terms in the operator $P(iD)$, then obviously the change of variables introduced above does not work. However, provided that we have good estimates on the symbol $P(\omega)$, it may be possible to show that $\varphi_t(\omega) = c_t e^{-tP(\omega)}$ defines an approximate identity nonetheless. In fact, we will encounter a similar situation in the proof of Theorem~\ref{thm:BHE-cyl}.
\end{enumerate}

\section{Biharmonic heat equation on infinite cylinders}
\label{sec:cyl}

In this section we prove Theorem~\ref{thm:BHE-cyl}. The analysis is more technical than the problem on $\RR^N$, and we will need a handful of preparatory results before proceeding with the main proof.

\subsection{A parametrised family of elliptic operators}

A common approach to solve the heat equation on $\RR\times\Omega$ is to use separation of variables. If we take an initial function with the special form $u_0(x,y) = f(x)g(y)$, we can build a solution to the evolution problem by solving $v_t = v_{xx}$ with $v(0, x) = f(x)$ on $\RR$, and $w_t = \Delta_y w$ with $w(0, y) = g(y)$ on $\Omega$ and $w(t, \cdot) = 0$ on $\partial\Omega$ for all $t \geq 0$. Here, as in the sequel, $\Delta_y$ denotes the Laplacian in only the $y$ variable. Then it is easy to see that
\begin{equation*}
  u(t,x,y) := v(t,x) w(t, y)
\end{equation*}
solves $u_t = \Delta u$ on $\RR \times \Omega$ with initial condition $u_0(x,y)$ and Dirichlet boundary condition $u(t, x, \cdot) = 0$ on $\partial\Omega$. In contrast to the second order equation, a similar separation of variables is not possible for the biharmonic operator.

Nevertheless, we proceed by taking the partial Fourier transform of problem~\eqref{eq:BHE-cyl} with respect to the $x$ variable. We obtain a family of functions $\eta(t, \omega, y) := \widehat{u}(t, \omega, y)$ such that for each $\omega \in \RR$, the function $\eta$ solves the equation
\begin{equation}
  \label{eq:eta-eq}
  \begin{aligned}
    \frac{\partial \eta}{\partial t} & = -\Delta_y^2 \eta + 2 \omega^2 \Delta_y \eta - \omega^4 \eta
                                     &                                                               & \text{on } (0, \infty) \times \Omega,                                                          \\
    \eta(0, \omega, y)               & = \widehat{u}_0(\omega,y)
                                     &                                                               & \text{for } y \in \Omega,                                                                      \\
    \eta(t,\omega,\cdot)             & = \frac{\partial \eta}{\partial \nu}(t,\omega,\cdot) = 0      &                                       & \text{on } \partial \Omega \text{ for each } t \geq 0,
  \end{aligned}
\end{equation}
where
\begin{equation*}
  \widehat{u}_0(\omega,y) = \frac{1}{\sqrt{2\pi}}\int_\RR u_0(x,y) e^{-i \omega x} \,d x
\end{equation*}
is the partial Fourier transform of the initial datum in the $x$ variable. By the vector-valued versions of the Riemann-Lebesgue Lemma and Plancherel's theorem from~\cite[Theorems~1.8.1, and~1.8.2]{ABHN} and the assumption~\eqref{eq:BHE-cyl-u0}, we conclude that
\begin{equation}
  \label{eq:u0-hat}
  \widehat{u}_0 \in C_0 \cap L^2(\RR, L^2(\Omega)).
\end{equation}
In particular we deduce that
\begin{equation}
  \label{eq:u0-hat-bounded}
  M:=\sup_{\omega \in \RR} \| \widehat{u}_0(\omega, \cdot) \|_{L^2(\Omega)} < \infty
\end{equation}
and that
\begin{equation}
  \label{eq:u0-hat-plancerel}
  \int_\RR \| \widehat{u}_0(\omega, \cdot) \|^2_{L^2(\Omega)} \, d\omega
  = 2\pi\int_\RR \|u(x, \cdot)\|^2_{L^2(\Omega)} \, dx
  < \infty.
\end{equation}
For each $\omega\in\RR$, let $(-\mu_n(\omega))_{n \geq 1}$ be the family of eigenvalues  of the operator
\begin{equation}
  \label{eq:L-omega}
  L_\omega := -\Delta_y^2 +2 \omega^2 \Delta_y - \omega^4
\end{equation}
with domain
\begin{equation*}
  H^4 \cap H^2_0(\Omega).
\end{equation*}
We take these eigenvalues in increasing order counting multiplicities. We also choose a family of corresponding eigenfunctions $(\phi_n(\omega, \cdot))_{n \in \NN}$ forming an orthonormal basis of $L^2(\Omega)$. At the moment we do this for fixed $\omega$, but we will later need some regularity of $\mu_n$ and $\phi_n$ in $\omega$. This will be established in Proposition~\ref{prop:cont-phi-1}. Hence, the solution of~\eqref{eq:eta-eq} can be represented by the Fourier series
\begin{equation}
  \label{eq:eta-Fourier-series}
  \eta(t, \omega, y) = \sum_{n=1}^\infty e^{-t \mu_n(\omega)} A_n(\omega) \phi_n(\omega, y),
\end{equation}
where the coefficients are given by
\begin{equation}
  \label{eq:An}
  A_n(\omega) = \int_\Omega \widehat{u}_0(\omega,y) \phi_n(\omega, y) \,dy.
\end{equation}
In the sequel, it will be useful to consider the operator $-L_\omega$ instead. Then for each fixed $\omega \in \RR$, the eigenvalues of $-L_\omega$ satisfy $\mu_n(\omega) > 0$ for all $n \geq 1$. In fact, as we will show in Lemma~\ref{lem:mu-omegas}(iii), it holds that $0 < \lambda_1^2 \leq \mu_n(\omega)$ for all $n \ge 1$ and all $\omega \in \RR$, where $\lambda_1$ is the principal eigenvalue of the Dirichlet Laplacian. The bilinear form associated with $-L_\omega$ is given by
\begin{equation}
  \label{eq:a-omega}
  \mfra_\omega(u,v) := \int_\Omega \Delta u \, \Delta v \,dy + 2\omega^2 \int_\Omega \nabla u \cdot \nabla v \,dy + \omega^4 \int_\Omega uv \,dy
\end{equation}
defined for all $u, v \in H^2_0(\Omega)$. Note that $-L_0 = \Delta^2_y$ is the biharmonic operator. We will drop the subscript $y$ for notational convenience.

\subsection{Analysis of the parametrised elliptic equations}

We begin with a simple, general result for the Sobolev space $H^2_0(\Omega)$. The notation $D^2 u$ denotes the Hessian matrix of $u$.
\begin{proposition}
  \label{prop:H22-norm}
  Let $\Omega \subseteq \RR^N$ be a bounded open set. Then the quantity
  \begin{equation*}
    |u|_{2,2} := \|\Delta u\|_{L^2(\Omega)}
  \end{equation*}
  defines a norm on $H^2_0(\Omega)$ that is equivalent to the usual $H^2$ norm:
  \begin{equation*}
    \|u\|_{H^2(\Omega)} = \left(\|u\|^2_{L^2(\Omega)} + \|\nabla u\|^2_{L^2(\Omega)} + \|D^2 u\|^2_{L^2(\Omega)} \right)^{1/2},
  \end{equation*}
  where $\|D^2u\|_{L^2(\Omega)}$ is the Hilbert-Schmidt or Frobenius norm of the Hessian matrix.
\end{proposition}

\begin{proof}
  Firstly, we show that $\|D^2 u\|_{L^2(\Omega)} = \| \Delta u \|_{L^2(\Omega)}$ for all $u \in C^\infty_0(\Omega)$. Let $\partial_j u$ ($j=1, \ldots, N$) denote any of the partial derivatives of $u$. Then, using integration by parts twice, we obtain
  \begin{multline*}
    \|D^2 u \|^2_{L^2(\Omega)} = \sum_{i,j = 1}^N \int_\Omega (\partial_i \partial_j u)^2 \,dy
    = \sum_{i,j = 1}^N \int_\Omega (\partial_{ii} u)(\partial_{jj}u) \,dy \\
    = \int_\Omega (\Delta u)^2 \,dy = \|\Delta u\|^2_{L^2(\Omega)}.
  \end{multline*}
  Now the above identity extends by density to all $u \in H^2_0(\Omega)$.

  If $u \in H^2_0(\Omega)$, then $\partial_j u \in H^1_0(\Omega)$ for each $j=1, \ldots, N$. Using the Poincar\'{e} inequality on $\partial_j u$, we find that there exists a constant $C_1 = C_1(\Omega)$ such that
  \begin{equation*}
    \| \nabla u \|_{L^2(\Omega)} \leq C_1 \| D^2 u \|_{L^2(\Omega)} = C_1 \| \Delta u \|_{L^2(\Omega)}.
  \end{equation*}
  Applying the Poincar\'{e} inequality to $u$, we obtain $\|u\|_{L^2(\Omega)} \leq C_2 \|\nabla u\|_{L^2(\Omega)}$ for another constant $C_2 = C_2(\Omega)$. In conclusion we obtain $C = C(\Omega)$ such that
  \begin{equation*}
    \|u\|_{H^2(\Omega)} \leq C \|D^2 u\|_{L^2(\Omega)} = C \|\Delta u\|_{L^2(\Omega)} = C|u|_{2,2}.
  \end{equation*}
  Finally, the fact that $|u|_{2,2}$ is a norm follows from well-posedness of the second-order Dirichlet problem.
\end{proof}

\begin{remark}
  \label{rmk:quad}
  Proposition~\ref{prop:H22-norm} shows the coercivity of the bilinear form $\mfra_0(u,v) = \int_\Omega \Delta u \, \Delta v \,dy$ defined for all $u, v \in H^2_0(\Omega)$. Since $\mfra_0(u,u) \leq \mfra_{\omega}(u,u)$ from the definition~\eqref{eq:a-omega} for all $\omega \in \RR$, it follows that each of the bilinear forms $\mfra_\omega(\cdot,\cdot)$ are coercive. Furthermore, we have the estimate
  \begin{multline}
    \label{eq:D2phi-estimate}
    \| \Delta \phi_n(\omega, \cdot) \|_{L^2(\Omega)}
    = \mfra_0(\phi_n(\omega, \cdot), \phi_n(\omega, \cdot))^{1/2}\\
    \leq \mfra_\omega(\phi_n(\omega, \cdot), \phi_n(\omega, \cdot))^{1/2}
    = \mu_n(\omega)^{1/2}
  \end{multline}
  if we choose the eigenfunctions $\phi_n(\omega,\cdot)$ to be normalised in $L^2(\Omega)$.
\end{remark}

The eigenvalues of the biharmonic operator play a distinguished role in the sequel. Consequently we define
\begin{equation*}
  \alpha_n := \mu_n(0) \qquad \text{for each integer } n \ge 1.
\end{equation*}
Now we present a comparison result between the eigenvalues of $-L_0$, $-L_\omega$ and those of the Dirichlet Laplacian. They will be used frequently in the main proof of this section.

\begin{lemma}
  \label{lem:mu-omegas}
  Let the differential operators $-L_\omega$ be defined as in~\eqref{eq:L-omega} with corresponding eigenvalues $(\mu_n(\omega))_{n \ge 1}$ in ascending order counting multiplicity. Set $\alpha_n := \mu_n(0)$. For each integer $n \ge 1$, the following assertions hold:
  \begin{enumerate}[\upshape (i)]
  \item $\mu_n\in C(\RR)$ is positive, even and strictly increasing as a function of $|\omega|$.
  \item For every $\omega \in \RR$ and $n\geq 1$, we have
    \begin{equation*}
      0
      <\alpha_n + \omega^4
      \leq \mu_n(\omega)
      \leq \alpha_n + 2 \alpha_n^{1/2} \omega^2 + \omega^4.
    \end{equation*}
  \item If $(\lambda_n)_{n\geq 1}$ are the eigenvalues of the Dirichlet Laplacian on $\Omega$, then
    \begin{equation}
      \label{eq:DL-biharmonic}
      \lambda_n^2 \leq \alpha_n \qquad \text{for all } n \ge 1.
    \end{equation}
  \end{enumerate}
\end{lemma}

\begin{proof}
  The proof relies on variational characterisations for eigenvalues of self-adjoint elliptic operators. Let $\mfrq(\cdot)$ be the quadratic form associated with the Dirichlet Laplacian, given by
  \begin{equation*}
    \mfrq(u):=\|\nabla u\|_{L^2(\Omega)}^2
  \end{equation*}
  for all $u\in H_0^1(\Omega)$. Furthermore, let $\mfra_0(\cdot)$ and $\mfra_\omega(\cdot)$ denote the quadratic forms associated with $-L_0$ and $-L_\omega$ respectively. They are given by
  \begin{equation*}
    \begin{aligned}
      \mfra_0(u)      & := \|\Delta u\|_{L^2(\Omega)}^2                                   \\
      \mfra_\omega(u) & := \mfra_0(u) + 2 \omega^2\mfrq(u)+\omega^4 \|u\|_{L^2(\Omega)}^2
    \end{aligned}
  \end{equation*}
  for all $u\in H_0^2(\Omega)$.

  (i) For $\gamma\in\RR$ consider the quadratic form
  \begin{equation*}
    \mfrb_\gamma(u) := \mfra_0(u)+ 2 \gamma\mfrq(u)
  \end{equation*}
  on $H_0^2(\Omega)$. It is the quadratic form associated with the operator $(-\Delta)^2 u+2\gamma\Delta u$ with clamped boundary conditions. According to the min-max and max-min principles for the the $n$-th eigenvalue $\nu_n(\gamma)$, we have
  \begin{equation}
    \label{eq:ev-minmax-maxmin}
    \nu_n(\gamma)
    =\min_{\dim(M)=n}\Bigl[\max_{\substack{u\in M\\\|u\|_{L^2}=1}}\mfrb(u)\Bigr]
    =\max_{\codim(M)=n-1}\Bigl[\min_{\substack{u\in M\\\|u\|_{L^2}=1}}\mfrb(u)\Bigr],
  \end{equation}
  where $M$ are subspaces of $H_0^2(\Omega)$, see for instance \cite[Chapters~2 and~3]{WS72}. In particular, $\nu_n$ is increasing as a function of $\gamma\in\RR$. Next note that the infimum and supremum of affine functions on $\RR$ are concave and convex, respectively. As convex and concave functions are continuous,
  \begin{equation*}
    f_M(\gamma):=\max_{\substack{u\in M\\\|u\|_{L^2}=1}}\mfrb_\gamma(u)
    \qquad\text{and}\qquad
    g_M(\gamma):=\min_{\substack{u\in M\\\|u\|_{L^2}=1}}\mfrb_\gamma(u)
  \end{equation*}
  are both continuous functions of $\gamma\in\RR$ for every relevant subspace $M\subseteq H_0^2(\Omega)$. We conclude from \eqref{eq:ev-minmax-maxmin} that $\nu_n$ is the supremum and an infimum of continuous functions. It is well-known that a supremum of continuous functions is lower semi-continuous. Likewise, an infimum is upper semi-continuous. Hence $\nu_n\in C(\RR)$. Finally observe that $\mfra_\omega(u)=\mfrb_{\omega^2}(u)+\omega^4\|u\|_{L^2(\Omega)}^2$ and thus $\mu_n(\omega)=\nu_n(\omega^2)+\omega^4$. Hence $\mu_n\in C(\RR)$ is symmetric, strictly increasing and unbounded as a function of $|\omega|$.

  (ii) Using integration by parts and the Cauchy-Schwarz inequality, note that
  \begin{equation}
    \label{eq:q-a-estimate}
    \mfrq(u)
    =\int_\Omega (-\Delta u)u\,dy
    \leq \Bigl(\int|\Delta u|^2\,dy\Bigr)^{1/2}\|u\|_{L^2(\Omega)}
    =\mfra_0(u)^{1/2}\|u\|_{L^2(\Omega)}^2
  \end{equation}
  for all $u\in H_0^2(\Omega)$. Hence, if $u\in H_0^2(\Omega)$ with $\|u\|_{L^2(\Omega)}=1$, then
  \begin{equation*}
    \mfra_0(u)+\omega^4
    \leq \mfra_\omega(u)
    \leq \mfra_0(u)+2\omega^2\mfra_0(u)^{1/2}+\omega^4
    =\left[\mfra_0(u)^{1/2}+\omega^2\right]^2.
  \end{equation*}
  Given a subspace $M\subseteq H_0^2(\Omega)$ with $\codim(M)=n-1$ we have
  \begin{equation}
    \label{eq:inf-left}
    \min_{\substack{u\in M\\\|v\|_{L^2}=1}}\mfra_0(u)+\omega^4
    \leq\min_{\substack{u\in M\\\|u\|_{L^2}=1}}\mfra_\omega(u)
  \end{equation}
  and
  \begin{equation}
    \label{eq:inf-right}
    \min_{\substack{u\in M\\\|v\|_{L^2}=1}}\mfra_\omega(u)
    \leq\min_{\substack{v\in M\\\|u\|_{L^2}=1}}
    \left[\mfra_0(u)^{1/2}+\omega^2\right]^2
    =\Bigl[\min_{\substack{v\in M\\\|u\|_{L^2}=1}}\mfra_0(u)^{1/2}+\omega^2\Bigr]^2,
  \end{equation}
  where the last equality holds since the function $s\mapsto (s^{1/2}+\omega^2)^2$ is continuous and increasing for $s\in[0,\infty)$. Taking the supremum in \eqref{eq:inf-left} and \eqref{eq:inf-right} over all subspaces $M\subseteq H_0^2(\Omega)$ with $\codim(M)=n-1$, the maximum-minimum characterisation of eigenvalues gives
  \begin{equation*}
    \alpha_n+\omega^4\leq \mu_n(\omega)
    \leq\left[\alpha_n^{1/2}+\omega^2\right]^2
    =\alpha_n+2\alpha_n^{1/2}\omega^2+\omega^4,
  \end{equation*}
  where for the last inequality we used again that $s\mapsto (s^{1/2}+\omega^2)^2$ is continuous and increasing.

  (iii)  Let $\Lambda_n$ and $\Sigma_n$ denote the set of all $n$-dimensional subspaces of $H^1_0(\Omega)$ and $H^2_0(\Omega)$ respectively. Observe that $\Sigma_n\subseteq\Lambda_n$. Hence, by~\eqref{eq:q-a-estimate} and the minimum-maximum principle,
  \begin{equation*}
    \alpha_n
    = \min_{M \in \Sigma_n} \Bigl[\max_{\substack{u\in M\\\|u\|_{L^2}=1}}\mfra_0(u)\Bigr]
    \geq \min_{M \in \Sigma_n} \Bigl[\max_{\substack{u\in M\\\|u\|_{L^2}=1}}\mfrq(u)^2\Bigr]
    \geq \min_{M \in \Lambda_n} \Bigl[\max_{\substack{u\in M\\\|u\|_{L^2}=1}}\mfrq(u)^2\Bigr]
    =\lambda_n^2
  \end{equation*}
  for all $n \ge 1$.
\end{proof}

\begin{remark}
  A different proof of~\eqref{eq:DL-biharmonic} for $n=1$ may be found in~\cite[Remark 4]{SS20}.
\end{remark}

The following result is a consequence of a well-known theorem of Weyl and the above lemma. Note that for sequences of real numbers $(a_n)_{n \ge 1}$ and $(b_n)_{n \ge 1}$, we write $a_n \sim b_n$ to mean that $\lim_{n \to \infty} \dfrac{a_n}{b_n} = 1$.
\begin{corollary}
  \label{cor:eig-series}
  Let $(\alpha_n)_{n \ge 1}$ be the eigenvalues of the biharmonic operator $-L_0 = \Delta^2$ with clamped boundary conditions on $\Omega \subseteq \RR^N$. Then for every $k>N/4$
  \begin{equation}
    \label{eq:eig-series2}
    \sum_{n=1}^\infty \frac{1}{\alpha_n^{k}} < \infty.
  \end{equation}
\end{corollary}

\begin{proof}
  The Weyl asymptotic law for the Dirichlet Laplacian states that
  \begin{equation*}
    \lim_{n \to \infty} \frac{\lambda_n}{n^{2/N}}
    = \frac{4\pi^2}{(B_N |\Omega|)^{2/N}}
  \end{equation*}
  where $B_N$ is the volume of the unit $N$-ball and $|\Omega|$ is the volume of $\Omega$, see for instance \cite[p.~55]{ANPS}. Since $\sum_{n=1}^\infty n^{-4k/N}<\infty$ whenever $k>N/4$, the Weyl law and thus $\sum_{n=1}^\infty\lambda_n^{-2k}$ also converges for $k>N/4$. Hence, by Lemma~\ref{lem:mu-omegas}(iii)
  \begin{equation}
    \label{eq:eig-series-dom}
    \sum_{n=1}^\infty\frac{1}{\alpha_n^k}
    \leq\sum_{n=1}^\infty \frac{1}{\lambda_n^{2k}}<\infty
  \end{equation}
  as claimed.
\end{proof}

\begin{remark}
  In the case $N=1$, a much more computational proof is possible. For simplicity, we take $\Omega = (0,1)$. It is known (for example, see~\cite[p.\,296]{CH}) that the eigenvalues of $-L_0$ for the clamped boundary conditions on $[0, 1]$ are given by
  \begin{equation}
    \alpha_n := k_n^4 \qquad \text{where } \cos(k_n)\cosh(k_n) = 1 \quad n \in \NN.
  \end{equation}
  The equation $\cos (k_n) \cosh (k_n) = 1$ is equivalent to $\cosh(k_n) = \sec(k_n)$. On the positive half-line, the secant function has vertical asymptotes at $\tfrac{(2m+1)\pi}{2}, m = 0, 1, 2, \ldots$, and it is positive on the intervals
  \begin{equation*}
    \left(\tfrac{(2m+1)\pi}{2}, \tfrac{(2m+3)\pi}{2} \right) \qquad m = 1, 3, 5, \ldots
  \end{equation*}
  We can rewrite the above expression in the form
  \begin{equation*}
    J_n := \left(\tfrac{(4n-1)\pi}{2}, \tfrac{(4n+1)\pi}{2} \right) \qquad n=1,2,3, \ldots
  \end{equation*}
  It follows that each $k_n$ lies in the corresponding $J_n$, and therefore
  \begin{equation*}
    \sum_{n=1}^\infty \frac{1}{\alpha_n} = \sum_{n=1}^\infty \frac{1}{k_n^4} \leq \frac{16}{\pi^4} \sum_{n=1}^\infty \frac{1}{(4n-1)^4}.
  \end{equation*}
  The latter series is evidently convergent.
\end{remark}

\begin{remark}
  Consider the solution operator for the equation $\Delta^2 u = f$ with boundary conditions $u = \frac{\partial u}{\partial \nu} = 0$ on $\partial\Omega$, defined by mapping $f \in L^2(\Omega)$ to the solution $u$. It can be realised as an integral operator, and the associated kernel is known as the \emph{Green's function} $G(x,y)$, which satisfies
  \begin{equation*}
    Tf := u(x) = \int_\Omega G(x,y)f(y) \,dy.
  \end{equation*}
  The eigenvalues of $T$ are precisely the reciprocals of the eigenvalues of $\Delta^2$. The convergence of the series $\sum_{n=1}^\infty \alpha_n^{-2}$ shows that $T$ is a \emph{Hilbert-Schmidt operator} for dimensions $N=1,\ldots,7$ (one can apply~\cite[Theorem 4.5]{Hal} to the orthonormal basis of eigenfunctions of $\Delta^2$). We also note that Weyl-type asymptotics hold for the biharmonic operator on very general domains. For example, the result
  \begin{equation}
    \frac{C_{N,\Omega}}{\alpha_n} \sim \frac{1}{n^{4/N}} \qquad \text{as } n \to \infty
  \end{equation}
  was shown by Levine and Protter in~\cite{LP} with the explicit constant
  \begin{equation*}
    C_{N,\Omega} = \frac{N}{N+4} 16\pi^4 (B_N|\Omega|)^{-4/N}.
  \end{equation*}
  We could have used this directly in the proof of Corollary~\ref{cor:eig-series}, but we think it is also interesting to deduce the result using only the well-known classical Weyl law for the Dirichlet Laplacian.
\end{remark}

As an application of the results thus far, we show that the function defined by the Fourier series~\eqref{eq:eta-Fourier-series} belongs to $L^1(\RR, L^2(\Omega))$ for each $t > 0$. This will allow us to represent the solution $u(t,x,y)$ to the problem~\eqref{eq:BHE-cyl} as the inverse Fourier transform of~\eqref{eq:eta-Fourier-series}.

\begin{proposition}
  \label{prop:eta-series}
  Let $\eta(t, \omega, y)$ be defined by~\eqref{eq:eta-Fourier-series}. Then
  \begin{equation}
    \label{eq:eta-Fourier-series-estimate}
    \int_\RR \| \eta(t, \omega, \cdot) \|_{L^2(\Omega)} \,d\omega
    \leq 2\pi\int_{\RR}\|u_0(\omega,\cdot)\|_{L^2(\Omega)}\,d\omega
  \end{equation}
  and
  \begin{equation}
    \label{eq:eta-Fourier-series-estimate-sup}
    \sum_{n=1}^\infty e^{-2t\mu_n(\omega)} |A_n(\omega)|^2
    \leq \sup_{\omega\in\RR}\|\widehat{u}_0(\omega,\cdot)\|_{L^2(\Omega)}<\infty
  \end{equation}
  for all $t\geq 0$.
\end{proposition}
\begin{proof}
  Fix $\omega \in \RR$. By Lemma~\ref{lem:mu-omegas}(i), we have for each $n \in \NN$ and $\omega\in\RR$
  \begin{equation*}
    \mu_n(\omega) \ge \alpha_n + \omega^4\geq \alpha_n>0.
  \end{equation*}
  Since the eigenfunctions $(\phi_n(\omega, \cdot))_{n \ge 1}$ were chosen to form an orthonormal basis of $L^2(\Omega)$, by Parseval's identity for Fourier series we obtain
  \begin{equation*}
    \label{eq:eta-series1}
    \| \eta(t, \omega, \cdot) \|_{L^2(\Omega)}^2
    =\sum_{n=1}^\infty e^{-2t\mu_n(\omega)} |A_n(\omega)|^2
    \leq\sum_{n=1}^\infty |A_n(\omega)|^2
    =\|\widehat{u}_0(\omega,\cdot)\|_2^2.
  \end{equation*}
  Integrating this inequality and taking into account~\eqref{eq:u0-hat-plancerel} yields~\eqref{eq:eta-Fourier-series-estimate}. Inequality~\eqref{eq:eta-Fourier-series-estimate-sup} follows from~\eqref{eq:u0-hat-bounded}.
\end{proof}

As another application of Lemma~\ref{lem:mu-omegas}, we prove the following essential $L^\infty$-estimate on the eigenfunctions $\phi_n(\omega,\cdot)$.
\begin{lemma}
  \label{lem:supnorm}
  Let $\phi_n(\omega, \cdot)$ be an eigenfunction of the operator $-L_\omega$ corresponding to the eigenvalue $\mu_n(\omega)$. Then for every integer $k>N/4$, there exists a constant $C = C(N, k, \Omega) > 0$ such that
  \begin{equation}
    \label{eq:supnorm}
    \| \phi_n(\omega, \cdot) \|_{L^\infty(\Omega)}
    \leq C \left[1+(\alpha_n^{1/2}+\omega^2)^2\right]^k\|\phi_n\|_{L^2(\Omega)}
  \end{equation}
  for all $\omega\in\RR$ and all $n \ge 1$.
\end{lemma}

\begin{proof}
  We apply a bootstrap argument using standard regularity theory. Consider the Dirichlet boundary value problem
  \begin{equation}
    \label{eq:BH-bvp}
    \begin{aligned}
      \Delta^2 u & = f                                   &  & \text{in } \Omega \\
      u          & = \frac{\partial u}{\partial \nu} = 0 &  & \text{on }
      \partial\Omega
    \end{aligned}
  \end{equation}
  with $f \in L^2(\Omega)$.  The regularity theory for~\eqref{eq:BH-bvp} yields that if $f \in H^{2(m-1)}(\Omega)$ for some integer $m \ge 0$, then the solution $u \in H^{2(m+1)}(\Omega)$, and there exists a constant $C_\Omega$ depending on $\Omega$ and $m$ such that
  \begin{equation}
    \label{eq:elliptic-reg}
    \|u\|_{H^{2(m+1)}(\Omega)} \leq C_\Omega \|f\|_{H^{2(m-1)}(\Omega)};
  \end{equation}
  see~\cite[Corollary 2.21]{GGS}. Turning to the equation $-L_\omega[\phi_n] = \mu_n(\omega) \phi_n$, suppressing the dependence of $\phi_n$ on $\omega$ for convenience, we can rewrite it in the form
  \begin{equation}
    \label{eq:BH-bvp-rearranged}
    \Delta^2 \phi_n
    = [\mu_n(\omega) - \omega^4] \phi_n + 2 \omega^2 \Delta \phi_n.
  \end{equation}
  Assuming that $\phi_n\in H^{2m}(\Omega)$ for some integer $m\geq 1$, an application of~\eqref{eq:elliptic-reg} yields
  \begin{equation}
    \label{eq:phi-H4-1}
    \| \phi_n \|_{H^{2(m+1)}(\Omega)}
    \leq C_\Omega \left\|(\mu_n(\omega) - \omega^4) \phi_n + 2 \omega^2 \Delta \phi_n \right\|_{H^{2(m-1)}(\Omega)}.
  \end{equation}
  Now Lemma~\ref{lem:mu-omegas}(i) gives $0<\alpha_n\leq \mu_n(\omega) - \omega^4 \leq \alpha_n +2 \alpha_n^{1/2} \omega^2$ for all $\omega \in \RR$. We also note that $\|\phi_n\|_{H^{2(m-1)}(\Omega)}\leq \| \phi_n \|_{H^{2m}(\Omega)}$ and $\|\Delta\phi_n\|_{H^{2(m-1)}(\Omega)}\leq \| \phi_n \|_{H^{2m}(\Omega)}$. Combining these estimates with~\eqref{eq:phi-H4-1} we obtain
  \begin{equation}
    \label{eq:phi-H4-2}
    \| \phi_n \|_{H^{2(m+1)}(\Omega)}
    \leq C_\Omega \left[ \alpha_n + 2\alpha_n^{1/2}\omega^2 + 2\omega^2\right]\| \phi_n \|_{H^{2m}(\Omega)}.
  \end{equation}
  Since $2\omega^2\leq 1+\omega^4$ and $\alpha_n>0$ we can write
  \begin{equation*}
    \alpha_n + 2\alpha_n^{1/2}\omega^2 + 2\omega^2
    \leq  1+\alpha_n + 2\alpha_n^{1/2}\omega^2 + \omega^4
    =1+(\alpha_n^{1/2}+\omega^2)^2.
  \end{equation*}
  Recall that \emph{a priori} $\phi_n\in H^2(\Omega)$ and thus the right hand side of~\eqref{eq:BH-bvp-rearranged} is in $L^2(\Omega)$. Hence, starting with $m=1$, we can inductively apply~\eqref{eq:phi-H4-2} to obtain
  \begin{equation}
    \label{eq:phi-h4-m}
    \| \phi_n \|_{H^{2(m+1)}(\Omega)}
    \leq C_\Omega^{m-1} \bigl[1+(\alpha_n^{1/2}+\omega^2)^2\bigr]^{m} \| \phi_n \|_{H^2(\Omega)}
  \end{equation}
  for every $m\geq 1$. By Proposition~\ref{prop:H22-norm} and Remark~\ref{rmk:quad} there exists $\widetilde{C}$ depending only on $\Omega$ such that $\| \phi_n \|_{H^2(\Omega)} \leq \widetilde{C} \mu_n(\omega)^{1/2}\| \phi_n \|_{L^2(\Omega)}$. Moreover, by Lemma~\ref{lem:mu-omegas}
  \begin{equation*}
    \mu_n(\omega)^{1/2}
    \leq [\alpha_n+2\alpha_n^{1/2}\omega^2+\omega^4]^{1/2}
    =\alpha_n^{1/2}+\omega^2
    \leq 1+(\alpha_n^{1/2}+\omega^2)^2.
  \end{equation*}
  Let now $k>N/4$ and set $m:=k-1$. Then $2(m+1)>N/2$, and the Sobolev embedding theorem implies the existence of a constant $C_m$ such that $\| \phi_n \|_{L^\infty(\Omega)} \leq C_m \| \phi_n \|_{H^{2(m+1)}(\Omega)}$ and $\phi_n \in C(\overline\Omega)$, see for instance~\cite[Theorem 7.26]{GT}. By combining the above estimates with~\eqref{eq:phi-h4-m}, we deduce~\eqref{eq:supnorm} for all $\omega\in\RR$ with $C := C_m C_\Omega^{m} \widetilde{C}$ depending only on $N$ and $\Omega$.
\end{proof}

\begin{remark}
  \label{rmk:complex-omega}
  In preparation for the proof of Proposition~\ref{prop:cont-phi-1}(ii) below, we show that an $L^\infty$ bound similar to~\eqref{eq:supnorm} holds for $\phi_1(\omega)$ where $\omega$ now varies in a complex neighbourhood of 0. However, it will not be necessary to give the precise dependence on $\omega$, since the objective is simply to show local boundedness. Equation~\eqref{eq:BH-bvp-rearranged} can be recast in the form
  \begin{displaymath}
    \Delta^2 \phi_1 = f(\omega)\phi_1 + g(\omega) \Delta\phi_1
  \end{displaymath}
  where $f, g\colon \CC \to \CC$ are continuous functions. By~\eqref{eq:elliptic-reg}, we estimate
  \begin{displaymath}
    \begin{aligned}
      \| \varphi_1 \|_{H^{2(m+1)}(\Omega)} & \leq C \|f(\omega) \Delta \phi_1+g(\omega)\Delta\phi_1\|_{H^{2(m-1)}(\Omega)}                         \\
                                           & \leq C \left( |f(\omega)| \|\phi_1\|_{} + |g(\omega)| \| \Delta\phi_1 \|_{H^{2(m-1)}(\Omega)} \right) \\
                                           & \leq C\bigl(|f(\omega)| + |g(\omega)|\bigr) \|\phi_1\|_{H^{2m}(\Omega)}.
    \end{aligned}
  \end{displaymath}
  Similarly to~\eqref{eq:phi-H4-2}, we thus obtain
  \begin{equation}
    \label{eq:phi-h4-m-complex}
    \| \phi_1 \|_{H^{2(m+1)}(\Omega)} \leq C |F(\omega)| \| \phi_1 \|_{H^{2m}(\Omega)}
  \end{equation}
  for some continuous function $F\colon \CC \to \CC$ and a constant $C$ depending only on $\Omega$. The $L^\infty$ estimate now follows by inductively applying~\eqref{eq:phi-h4-m-complex} and using the Sobolev embedding theorem as before.
\end{remark}

For later purposes we need to show that it is possible to choose an orthonormal system associated with $\mu_n(\omega)$ that depends regularly on $\omega$. We also show that $\omega\mapsto\mu_n(\omega)$ is not only continuous but piecewise real analytic, meaning that it is analytic on $\RR$ except possibly at a set of isolated points.

\begin{proposition}
  \label{prop:cont-phi-1}
  Let $(\mu_n)_{n\in\NN}$ be the family of eigenvalues as in Lemma~\ref{lem:mu-omegas}. Then the following assertions are true:
  \begin{enumerate}[\upshape (i)]
  \item The function $\mu_n$ is continuous and piecewise real analytic on $\RR$, and we can choose an orthonormal system $(\phi_n)_{n\in\NN}$ of corresponding eigenfunctions such that $\phi_n\in L^\infty(\RR, L^2(\Omega))$ is piecewise real analytic.
  \item If $\mu_1(0)$ is algebraically simple, then there exists there exists $\delta>0$ and choice of normalised eigenfunction $\phi_1(\omega,\cdot)$ corresponding to $\mu_1(\omega)$ such that
    \begin{equation*}
      \phi_1\in C([-\delta,\delta]\times\overline\Omega).
    \end{equation*}
  \end{enumerate}
\end{proposition}

\begin{proof}
  (i) The continuity of $\omega\mapsto\mu_n(\omega)$ is proved in Lemma~\ref{lem:mu-omegas}(i), so we only need to prove the piecewise analyticity of the family of eigenvalues and eigenfunctions.

  The family~\eqref{eq:a-omega} is clearly a holomorphic function defined on the constant domain $H_0^2(\Omega)$ for all $\omega\in\CC$. It is symmetric, bounded below and the domain is compactly embedded in $L^2(\Omega)$. Hence Rellich's perturbation theorem \cite[Theorem~II.10.1]{R69} and analytic continuation allows to represent the eigenvalues as real analytic functions $\beta_n\colon\RR\to\RR$ with $\beta_n(0)=\mu_n(0)=\alpha_n$. Similarly, there exists an orthogonal system of corresponding eigenfunctions given by analytic functions $\psi_n\colon\RR\to L^2(\Omega)$, $n\in\NN$.

  The curves $\beta_n$ in general cross and hence do not coincide with the family $\mu_n$ ordered by size. We show that these crossing points form a discrete set. First note that by the uniqueness theorem for analytic functions, two such curves have either at most finitely many points of intersections, or they coincide for all $\omega\in\RR$. Also, at any crossing point at $\beta_m(\omega_0)$, at most finitely many curves can meet. Indeed, we know that $\beta_m(\omega_0)$ has finite algebraic multiplicity, say $\ell$. The lower semi-continuity of the spectrum from \cite[Theorem IV.3.16]{KAT} shows that for $\omega$ in a neighbourhood of $\omega_0$, the eigenvalue $\beta_m(\omega_0)$ splits into at most $\ell$ eigenvalues. Hence at most $\ell$ curves can intersect $\beta_m$ at $\omega_0$. The argument also shows that such crossing points cannot accumulate at any point of $\beta_m$ and thus form a discrete set. Hence, for $\omega$ between any two crossing points, $\mu_n$ follows some $\beta_m$, with associated eigenfunctions $\psi_n$. Then $\phi_n = \psi_n / \|\psi_n\|_{L^2(\Omega)}$ is the required family of eigenfunctions satisfying $\phi_n \in L^\infty(\RR, L^2(\Omega))$.

  (ii) If we assume that $\alpha_1 := \mu_1(0)$ is algebraically simple, then according to (i) there exists $\delta>0$ such that we can choose $\omega\mapsto\phi_1(\omega)$ analytic as a function from $[-\delta,\delta]$ into $L^2(\Omega)$. In fact, Rellich's theorem even shows that $\phi_1$ is analytic in a complex neighbourhood $D$ of $\omega = 0$. Moreover, from the proof of Lemma~\ref{lem:supnorm} and Remark~\ref{rmk:complex-omega}, one sees that $\phi_1(\omega)\in C(\overline\Omega)$, and that $\|\phi_1(\omega)\|_{L^\infty(\Omega)}$ can be uniformly bounded in $D$. We will show that $\phi_1$ is even analytic from $D$ into $C(\overline\Omega)$, and to do so, we verify the conditions of Theorem 3.1 in \cite{AN00}.

  We take the Banach space $X = C(\overline\Omega)$, and we identify the space $W = L^2(\Omega)$ with its own dual. It follows that $X \hookrightarrow W \hookrightarrow X'$, and moreover $W$ is a \emph{separating} subset of $X'$ (see \cite[p.\ 787]{AN00}). Indeed, every $f \in X$ can be considered as an element of $W$, hence $\langle f,f\rangle = \|f\|^2_{L^2(\Omega)} \ne 0$ if $f \ne 0$. From the previous paragraph we have that $\phi_1\colon D \to X$ is locally bounded, and analytic from $D$ into $W$. It follows that $D \ni \omega \mapsto \langle \psi, \phi_1(\omega)\rangle \in \CC$ is analytic for all $\psi \in W$. Since $W$ separates $X$, \cite[Theorem 3.1]{AN00} is applicable and yields that $\phi_1\colon D \to X$ is analytic. In particular, $\omega\mapsto\phi_1(\omega)$ can be chosen to be a continuous function from $[-\delta,\delta]$ to $C(\overline\Omega)$.
\end{proof}

\begin{remark}
  We note that the proof of Proposition~\ref{prop:cont-phi-1} did not use the specific structure of the operators $-L_\omega$, and thus the results may be adapted to more general analytic families of operators.
\end{remark}

\subsection{Asymptotic behaviour on infinite cylinders}
All the ingredients needed for the main result are now in place.

\begin{proof}[Proof of Theorem~\ref{thm:BHE-cyl}]
  Taking the inverse Fourier transform of the Fourier series representation~\eqref{eq:eta-Fourier-series} of $\widehat{u}(t,\omega,y)$, the solution to~\eqref{eq:BHE-cyl} can be represented in the form
  \begin{equation*}
    u(t,x,y)
    = \frac{1}{\sqrt{2\pi}}\int_{\RR} \left[ \sum_{n=1}^\infty e^{-t \mu_n(\omega)} A_n(\omega) \phi_n(\omega, y) \right] e^{i \omega x} \,d\omega
  \end{equation*}
  which is well-defined due to Proposition~\ref{prop:eta-series}. Splitting off the first term, we obtain
  \begin{multline}
    \label{eq:u-split1}
    u(t,x,y)
    = \frac{1}{\sqrt{2\pi}}\int_{\RR} e^{-t \mu_1(\omega)} A_1(\omega) \phi_1(\omega, y) e^{i\omega x} \,d\omega \\
    + \frac{1}{\sqrt{2\pi}}\int_{\RR} \left[ \sum_{n=2}^\infty e^{-t \mu_n(\omega)} A_n(\omega) \phi_n(\omega, y) \right] e^{i \omega x} \,d\omega.
  \end{multline}
  We define
  \begin{equation}
    \label{eq:ct-cyl}
    c_t:=2\pi\left(\int_{\RR}e^{-t\mu_1(\omega)}\,d\omega\right)^{-1},
  \end{equation}
  which is well-defined by Lemma~\ref{lem:mu-omegas}. We prove that~\eqref{eq:BHE-cyl-blowup-limit} holds with this choice of blow-up factors.

  \paragraph{Step 1}
  Let us deal with the first term on the right hand side of~\eqref{eq:u-split1}. Define
  \begin{equation*}
    f(x,\omega,y):=\sqrt{2\pi}A_1(\omega)\phi_n(\omega,y)e^{i\omega\cdot x},
  \end{equation*}
  and let $\varphi_t$ be the approximate identity defined in Corollary~\ref{cor:approx-id-mu} below. Then
  \begin{multline*}
    \frac{c_t}{\sqrt{2\pi}}\int_{\RR} e^{-t \mu_1(\omega)} A_1(\omega) \phi_1(\omega, y) e^{i\omega x} \,d\omega\\
    =\sqrt{2\pi}\int_{\RR^N}\varphi_t(\omega)A_1(\omega) \phi_1(\omega, y) e^{i\omega x} \,d\omega
    =\bigl(\check\varphi_t*f(x\,,\cdot\,,y)\bigr)(0),
  \end{multline*}
  where, as in the proof of Theorem~\ref{thm:BHE-Rn}, $\check\varphi_t(\omega):=\varphi_t(-\omega)$ for all $\omega\in\RR$.  Proposition~\ref{prop:cont-phi-1} asserts that $\mu_1\in C(\RR)$, and that for any compact interval $I=[-r,r]$ there exists an interval $J = [-\delta,\delta]$ such that $\varphi\in C(J\times\overline{\Omega})$. Hence $f\colon I\times J\times\overline{\Omega} \to \CC$ is continuous and thus uniformly continuous. This means that $\omega\mapsto f(x,\omega,y)$ is continuous at $\omega=0$ uniformly with respect to $(x,y)\in I\times\overline{\Omega}$. It follows from Lemma~\ref{lem:approx-id} and Corollary~\ref{cor:approx-id-mu} that
  \begin{equation*}
    \lim_{t\to\infty}\frac{c_t}{\sqrt{2\pi}}\int_{\RR} e^{-t \mu_1(\omega)} A_1(\omega) \phi_1(\omega, y) e^{i\omega x} \,d\omega
    =\lim_{t\to\infty}\bigl(\check\varphi_t*f(x\,,\cdot\,,y)\bigr)(0)
    =f(x,0,y)
  \end{equation*}
  uniformly with respect to $(x,y)\in I\times\Omega$. By definition of $A_1(0)$, we finally see that
  \begin{multline*}
    f(x,0,y)
    =\sqrt{2\pi}A_1(0)\phi_1(0,y)
    =\sqrt{2\pi}\int_{\Omega}\hat u_0(0,\eta)e_1(\eta)\,d\eta\, e_1(y)\\
    =\int_{\RR}\int_{\Omega}u_0(\xi,\eta)e_1(\eta)\,d\eta\,d\xi\, e_1(y).
  \end{multline*}

  \paragraph{Step 2}
  We claim that the product of the second term on the right hand side of~\eqref{eq:u-split1} with $c_t$ converges to zero as $t \to \infty$ uniformly with respect to $(x,y) \in \RR \times   \Omega$. To do so, observe firstly that
  \begin{equation*}
    \left|A_n(\omega) \phi_n(\omega, y)\right| \leq |A_n(\omega)| \|\phi_n(\omega,\cdot)\|_\infty.
  \end{equation*}
  By~\eqref{eq:eta-Fourier-series-estimate-sup} in Proposition~\ref{prop:eta-series}, there exists $M>0$ such that $|A_n(\omega)|\leq M$ for all $\omega\in\RR$ and $n \in \NN$. Furthermore, by Lemma~\ref{lem:supnorm} and an elementary inequality, we find
  \begin{equation}
    \label{eq:phi-n-expand}
    \|\phi_n(\omega, \cdot) \|_{L^\infty(\Omega)}
    \leq C \left[ 1+(\alpha_n^{1/2} +\omega^2)^2 \right]^k
    \leq 3^{k}C\left[ 1+\alpha_n^{k} +\omega^{4k} \right]
  \end{equation}
  with $k>N/4$ and $C$ depending on $N$, $k$ and $\Omega$. We therefore have
  \begin{equation*}
    \Bigl|\int_{\RR} \Bigl[ \sum_{n=2}^\infty e^{-t \mu_n(\omega)} A_n(\omega) \phi_n(\omega, y) \Bigr] e^{i \omega x} \,d\omega\Bigr|
    \leq 3^{k}CM\int_{\RR}S(t, \omega) \,d\omega,
  \end{equation*}
  where
  \begin{equation*}
    S(t, \omega) := \sum_{n = 2}^\infty \left[ 1+\alpha_n^k +\omega^{4k} \right]e^{ -t\mu_n(\omega) }.
  \end{equation*}
  We show that
  \begin{equation}
    \label{eq:S-t-omega}
    c_t \int_\RR S(t, \omega) \,d\omega \longrightarrow 0
  \end{equation}
  as $t \to \infty$. Since $S(t, \omega)$ is a series with non-negative terms, we may interchange the summation and the integral to write
  \begin{equation}
    \label{eq:S-t-omega-2}
    \int_\RR S(t, \omega) \,d\omega
    = \sum_{n = 2}^\infty\int_\RR \left[1+\alpha_n^k+\omega^4\right] e^{-t \mu_n(\omega)} \,d\omega.
  \end{equation}
  We now multiply each term in~\eqref{eq:S-t-omega-2} by $c_t/2\pi$. Using that all integrands are even functions, we have a sum of expressions of the form
  \begin{equation*}
    \frac{ \int_\RR \omega^{2m} e^{-t \mu_n(\omega) } \,d\omega}{\int_\RR e^{-t\mu_1(\omega)} \,d\omega }
    = \frac{ \int_0^\infty \omega^{2m}e^{-t \mu_n(\omega) } \,d\omega}{\int_0^\infty e^{-t\mu_1(\omega)} \,d\omega }
  \end{equation*}
  with $m=0,2$. Applying Lemma~\ref{lem:mu-omegas} and Lemma~\ref{lem:f/g-lim} below, there exist $C,t_0\geq 1$ such that
  \begin{multline}
    \label{eq:key-ratio}
    \frac{ \int_\RR \omega^{2m} e^{-t \mu_n(\omega) } \,d\omega}{\int_\RR e^{-t\mu_1(\omega)} \,d\omega }
    \leq e^{-t(\alpha_n-\alpha_1)}\frac{ \int_0^\infty \omega^{2m} e^{-t\omega^4 }\,d\omega}{\int_0^\infty e^{-t(\beta\omega^2+\omega^4)}\,d\omega }\\
    \leq Ct^{\frac{1}{4}-\frac{1}{2m}}e^{-t(\alpha_n-\alpha_1)}
    \leq Ct^{1/4}e^{-t(\alpha_n-\alpha_1)}
  \end{multline}
  for all $t\geq t_0$ and $n\geq 2$, where $\beta:=2\alpha_1^{1/2}$. As $\alpha_n\geq\alpha_2>\alpha_1$ for all $n\in\NN$, an application of the exponential series yields
  \begin{equation*}
    e^{-t(\alpha_n-\alpha_1)}
    \leq\frac{\ell!}{(\alpha_n-\alpha_1)^\ell}\frac{1}{t^\ell}
    \leq\Bigl(\frac{\alpha_n}{\alpha_n-\alpha_1}\Bigr)^{\ell}\frac{\ell!}{\alpha_n^\ell}\frac{1}{t^\ell}
    \leq\Bigl(\frac{\alpha_2}{\alpha_2-\alpha_1}\Bigr)^{\ell}\frac{\ell!}{\alpha_n^\ell}\frac{1}{t^\ell}
  \end{equation*}
  for every $\ell\in\NN$, $n \ge 2$ and $t > 0$. Combining this with~\eqref{eq:S-t-omega-2} we see that
  \begin{equation}
    \label{eq:key-ratio-estimate}
    \frac{c_t}{2\pi} \int_\RR S(t, \omega) \,d\omega
    \leq C\Bigl(\frac{\alpha_2}{\alpha_2-\alpha_1}\Bigr)^{\ell}\frac{\ell!}{t^{\ell-1/4}}\sum_{n=2}^\infty \frac{2+\alpha_n^k}{\alpha_n^\ell}
  \end{equation}
  for all $t>t_0$. We choose $\ell\geq 1$ such that $\ell-k>N/4$. Then by Corollary~\ref{cor:eig-series} the series on the right hand side of~\eqref{eq:key-ratio-estimate} converges, and hence~\eqref{eq:S-t-omega} is valid as $t\to\infty$. As $S(t,\omega)$ is independent of $y\in\Omega$, this yields the convergence of the second term in~\eqref{eq:u-split1} to zero uniformly with respect to $y\in\Omega$ as $t\to\infty$. This completes the proof of the theorem.
\end{proof}
We conclude by proving a technical result used in the above proof.
\begin{lemma}
  \label{lem:f/g-lim}
  Suppose $n \geq 1$ is an integer and $\alpha\geq 0$. Furthermore let $p(x)$ be a polynomial of the form
  \begin{equation*}
    p(x) = \sum_{j=k}^m c_j x^{j}
  \end{equation*}
  with $m\geq k\geq 1$ and $c_m,c_k>0$. Define
  \begin{equation}
    f_\alpha(t) := \int_0^\infty x^\alpha e^{-tx^n} \,dx
    \quad\text{and}\quad
    g(t) := \int_0^\infty e^{-tp(x)} \,dx.
  \end{equation}
  Then there exist constants $C,t_0>0$ depending only on $\alpha$, $n$, $k$ and the coefficients of $p$ such that
  \begin{equation}
    \label{eq:f/g-bound}
    \frac{f_\alpha(t)}{g(t)}
    \leq Ct^{\frac{1}{k}-\frac{\alpha + 1}{n}}
  \end{equation}
  for all $t\geq t_0$.
\end{lemma}
\begin{proof}
  Applying the substitution $s=tx^n$ in the definition of $f_\alpha(t)$ we have
  \begin{equation}
    \label{eq:f-gamma}
    f_\alpha(t)
    = \frac{1}{n}t^{-\frac{\alpha+1}{n}}\int_0^\infty s^{\frac{\alpha+1}{n}-1}e^{-s}\,ds
    = \frac{1}{n}\Gamma\left(\frac{\alpha+1}{n}\right)t^{-\frac{\alpha+1}{n}}.
  \end{equation}
  To deal with $g(t)$ we note that by assumption
  \begin{equation*}
    p(x)=x^k(c_k+c_{k+1}x+\cdots+c_mx^{m-k})
  \end{equation*}
  with $k\geq 1$ and $c_k>0$. Hence there exists $x_0>0$ such that $c_k+c_{k+1}x+\cdots+c_mx^{m-k}\geq \beta:=c_k/2$ for all $x\in[0,x_0]$. It follows that
  \begin{equation*}
    g(t) \geq \int_0^{x_0} e^{-tp(x^2)} \,dx
    \geq \int_0^{x_0}e^{-t\beta x^k}\,dx.
  \end{equation*}
  With the substitution $s=t\beta x_0^k$ we see that
  \begin{equation*}
    \int_0^{x_0}e^{-tp(x^2)} \,dx
    =\frac{1}{k\beta^{1/k}}t^{-1/k}
    \int_0^{t\beta x_0^k}s^{1/k-1}e^{-s}\,ds.
  \end{equation*}
  Now note that
  \begin{equation*}
    \lim_{t\to\infty}\int_0^{t\beta x_0^k}s^{1/k-1}e^{-s}\,ds
    =\Gamma\left(\frac{1}{k}\right)
  \end{equation*}
  and thus there exists $t_0\geq 0$ such that
  \begin{equation*}
    g(t)\geq \int_0^{t\beta x_0^k}s^{1/k-1}e^{-s}\,ds
    \geq\frac{1}{2k\beta^{1/k}}\Gamma\left(\frac{1}{k}\right)t^{-1/k}
  \end{equation*}
  for all $t\geq t_0$. Combining this with~\eqref{eq:f-gamma} we see that there exist constants $C, t_0>0$ depending only on $\alpha$, $n$, $k$ and the coefficients of $p$ such that~\eqref{eq:f/g-bound} holds for all $t\geq t_0$.
\end{proof}

\begin{corollary}
  \label{cor:approx-id-mu}
  Let $\mu_1(\omega)$ be as defined in Lemma~\ref{lem:mu-omegas}. Then
  \begin{equation*}
    \varphi_t(\omega) := \left(\int_{\RR} e^{-t \mu_1(\eta)} \,d\eta \right)^{-1}e^{-t\mu_1(\omega)}
  \end{equation*}
  defines an approximate identity for $t \to \infty$ in the sense of Definition~\ref{def:approx-id} (where we take $N=1$ in this definition).
\end{corollary}
\begin{proof}
  By definition of $\varphi_t(\omega)$, it is obvious that $\varphi_t(\omega) \geq 0$ for all $t>0$ and $\omega \in \RR$, and that $\int_\RR \varphi_t \,d\omega = 1$. It remains to show that for every $\delta>0$
  \begin{equation}
    \label{eq:approx-id-mu}
    J_\delta(t):=\int_{|\omega| \ge \delta} \varphi_t(\omega) \,d \omega \longrightarrow 0 \qquad \text{as } t \to \infty.
  \end{equation}
  From Lemma~\ref{lem:mu-omegas}(i) we obtain
  \begin{equation*}
    e^{-t(\alpha_1 + \beta \omega^2 + \omega^4)} \leq e^{-t \mu_1(\omega)} \leq e^{-t(\alpha_1 + \omega^4)}
  \end{equation*}
  for all $\omega \in \RR$, where $\beta := 2 \alpha_1^{1/2}$. Since the function $\mu_1$ is even, it suffices to consider $\omega \ge 0$. The above inequalities imply
  \begin{equation*}
    J_\delta(t)
    = \frac{ \int_\delta^\infty e^{-t \mu_1(\omega)}\,d \omega }{ \int_0^\infty e^{-t \mu_1(\omega)} \,d \omega }
    \leq \frac{ \int_\delta^\infty e^{-t(\alpha_1 + \omega^4)} \,d\omega}
    { \int_0^\infty e^{-t(\alpha_1 + \beta \omega^2 + \omega^4)} \,d\omega }
    = \frac{ \int_\delta^\infty e^{-t \omega^4} \,d\omega }
    { \int_0^\infty e^{-t(\beta \omega^2 + \omega^4)} \,d\omega }.
  \end{equation*}
  As $e^{-t \omega^4} \leq \delta^{-2} \omega^2 e^{-t\omega^4}$ holds for all $\omega \in [\delta, \infty)$ and all $t>0$ we deduce that
  \begin{equation*}
    J_\delta(t)
    \leq \frac{1}{\delta^2}
    \frac{ \int_\delta^\infty \omega^2 e^{-t \omega^4} \,d\omega }
    { \int_0^\infty e^{-t(\beta \omega^2 + \omega^4)} \,d\omega }
    \leq \frac{1}{\delta^2}
    \frac{ \int_0^\infty \omega^2 e^{-t \omega^4} \,d\omega }
    { \int_0^\infty e^{-t(\beta \omega^2 + \omega^4)} \,d\omega }.
  \end{equation*}
  By Lemma~\ref{lem:f/g-lim}, there exists $C,t_0>0$ such that $J_\delta(t)\leq Ct^{\frac{1}{2}-\frac{3}{4}}=Ct^{-1/4}$ for all $t>t_0$. Hence $J_\delta(t)\to 0$ as $t\to\infty$ as required.
\end{proof}

\begin{remark}
  Observe that the decay rate $t^{-1/4}$ obtained in the above corollary is consistent with the $L^\infty$ decay of the solution on the full space---use~\eqref{eq:PHE-Rn-ct} with $\alpha = 2$ and $N=1$. In addition, the \emph{spectral gap} $\alpha_2 - \alpha_1 > 0$ of the biharmonic operator with clamped boundary conditions on the cross-section domain $\Omega$ appears explicitly in the uniform bound~\eqref{eq:key-ratio-estimate}. These features clearly show how the behaviour of the solution on the infinite cylinder is influenced by properties of the biharmonic operators respectively on the real line and on the bounded domain.
\end{remark}

\end{document}